\documentclass[a4paper,draft,12pt]{amsproc}
\usepackage{amssymb}
\usepackage[hyphens]{url} 
\usepackage{enumerate}
\usepackage[shortlabels]{enumitem}
\usepackage[T1]{fontenc}
\usepackage[latin1]{inputenc}
\usepackage{relsize}
\usepackage{xfrac}
\usepackage{floatrow}
\usepackage{amsthm}
\usepackage{tikz}
\usepackage{pgfplots} 
\usepackage{mathtools}
\usepackage{amsfonts}
\usepackage{amsmath}
\usepackage{faktor}
\usepackage{pict2e}
\usepackage{picture}
\usepackage{lmodern}
\usepackage{graphicx}
\usepackage{layout}
\usepackage{relsize}
\usepackage{calc}
\usepackage{comment} 
\usepackage{titlesec}
\usepackage[english]{babel}
\usepackage{bigints}
\usepackage{yhmath}

\titlelabel{\thetitle.\quad}
\newlength{\depthofsumsign}
\newlength{\totalheightofsumsign}
\newlength{\heightanddepthofargument}

\titleformat{\section}
{
\center\normalsize}
{\thesection.}{1em}{}

\titleformat{\subsection}
{\normalsize\bfseries}
{\thesubsection.}{1em}{}

\makeatletter
\newcommand*{\DivideLengths}[2]{%
  \strip@pt\dimexpr\number\numexpr\number\dimexpr#1\relax*65536/\number\dimexpr#2\relax\relax sp\relax
}
\makeatother

\newcommand{\N}{\mathbb{N}}
\newcommand{\Z}{\mathbb{Z}}

\newcommand{\R}{\mathbb{R}}
\newcommand{\C}{\mathbb{C}}
\newcommand{\D}{\mathbb{D}}
\newcommand{\T}{\mathbb{T}}
\newcommand{\A}{\mathcal{A}}
\newcommand{\B}{\mathcal{B}}

\usepackage{dsfont}

\newcommand{\norme}[1]{\left\Vert #1\right\Vert}
\newcommand{\normeinf}[1]{\norme{#1}_{\infty}}

\newcommand{\Dbar}{\overline{\D}}
\newcommand{\Hinf}[1]{H^{\infty}(#1\D)}
\newcommand{\dbfc}[1]{${#1}$-derivative bounded functional calculus}
\newcommand{\gdbfc}[1]{${#1}$-derivative $\gamma$-bounded functional calculus}

\theoremstyle{plain}
\newtheorem{thm}{Theorem}[section]
\newtheorem{prop}[thm]{Proposition}

\newtheorem{lem}[thm]{Lemma}

\newtheorem{cor}[thm]{Corollary}

\theoremstyle{definition}
\newtheorem{df}[thm]{Definition}
\theoremstyle{remark}

\newtheorem{rq1}[thm]{Remark}
\newtheorem{rqs}[thm]{Remarks}
\newtheorem{ex}[thm]{Example}

\numberwithin{equation}{section} 

\theoremstyle{break} 

\theoremstyle{nonumberplain}

\subjclass[2010]{Primary REQUIRED; Secondary OPTIONAL}

\title[Derivative bounded functional calculus of power bounded operators]{Derivative bounded functional calculus of power bounded operators on Banach spaces}
\subjclass[2010]{47A60, 46B28, 42B35}

\keywords{$\gamma$-boundedness, power bounded operators, functional calculus, Besov spaces}

\author[L.ARNOLD]{LORIS ARNOLD} 

\address{
	LABORATOIRE DE MATHÉMATIQUES DE BESANÇON, UMR 6623, CNRS \\ 
	UNIVERSITÉ DE FRANCHE-COMTÉ  \\ 
	25030 BESANÇON CEDEX\\
	FRANCE}
\email{loris.arnold@univ-fcomte.fr}

\thanks{This work is supported by the French ``Investissements d'Avenir" program, project ISITE-BFC (contract ANR-15-IDEX-03).} 

\begin{document}

{\begin{flushleft}\baselineskip9pt\scriptsize

\end{flushleft}}
\vspace{18mm} \setcounter{page}{1} \thispagestyle{empty}

\begin{abstract}
In this article we study bounded operators $T$ on Banach space $X$ which satisfy the discrete Gomilko Shi-Feng condition 	
	\begin{equation*}
\int_{0}^{2\pi}|\langle R(re^{it},T)^{2}x,x^*\rangle |dt \leq \frac{C}{(r^2-1)}\norme{x}\norme{x^*},\quad r>1, x\in X, x^* \in X^*.
\end{equation*}	
We show that it is equivalent to a certain derivative bounded functional calculus and also to a
 bounded functional calculus relative to Besov space. Also on Hilbert space discrete Gomilko Shi-Feng
  condition is equivalent to power-boundedness. Finally we discuss the last equivalence on general Banach space involving the concept of $\gamma$-boundedness. 
   
\end{abstract}
\maketitle

\section{Introduction}
Let $X$ be a Banach space, a bounded operator $T :X \rightarrow X$ is called
 power-bounded when $\underset{n \in \N}\sup \norme{T^n} < \infty$. In this
  case the spectrum of $T$ is contained in the closed unit ball $\Dbar$. We say that $T$ is polynomially bounded when it satisfies an estimate
\begin{equation}\label{polybounded}
\norme{P(T)} \leq C \sup\{ |P(z)|: z\in \D \}, \quad P \text{ polynomial. } 
\end{equation}
 When $X= H$ is an Hilbert space and $T$ is a contraction, that is
   $\norme{T}\leq 1$ (which is a power-bounded operator) it is well-known
    that $T$ satisfies \eqref{polybounded} with $C =1$. This is the so-called
     von Neumann's inequality. However even on Hilbert space we cannot expect
          that any power-bounded operator satisfies \eqref{polybounded} (see \cite{leb}). A natural
           question is whether one can obtain similar estimates, for power-bounded operators,
            replacing the norm uniform in the right-hand side of \eqref{polybounded} by an another function norm. An answer is given by Peller in
             \cite{pel1} for power-bounded operators on Hilbert space. If $T$ is such operator, then
              it satisfies
          \begin{equation}\label{polyboundedbesov}
          \norme{P(T)} \leq C\norme{P}_{\B(\D)} , \quad P \text{ polynomial. } 
          \end{equation}
         where $\norme{\cdot}_{\B(\D)}$ is an appropriate Besov norm (see section 3).

         In the first part of this paper we deal with operators which satisfy the discrete Gomilko
          Shi-Feng condition: an operator $T : X
            \rightarrow X$ satisfies this condition if the spectrum of $T$
             is included in $\Dbar$ and   
             
	\begin{equation}\label{strineprime}
\int_{0}^{2\pi}|\langle R(re^{it},T)^{2}x,x^*\rangle |dt \leq \frac{C}{(r^2-1)}\norme{x}\norme{x^*},\quad r>1, x\in X, x^* \in X^*.
\end{equation}
This condition was first introduced in \cite{gom2} and \cite{shifeng}.The continuous case had been introduced before in \cite{gom} and studied extensively in \cite{bat-haa}, \cite{arn1} and \cite{bgt1}.
We will show that this condition implies power boundedness and that the converse is true when $X =H$
 is an Hilbert space. We will be able to show (see section 3) that a bounded operator $T$ on a general Banach space $X$ satisfies an estimate \eqref{polyboundedbesov} if and only if $T$ satisfies the discrete Gomilko Shi-Feng condition \eqref{strineprime}. 
 Furthermore one of the main results of this paper is to show that the
  discrete Gomilko Shi-Feng condition is equivalent to the boundedness of the set 
  
 	\begin{equation}\label{primefuncalc}
\Big \{ (r-1)\phi'(T): r> 1 \text{ and } \phi \in H^{\infty}(r\D) \text{ with }\sup\{ |\phi(z)|: z\in r\D \} \leq 1 \Big\}.
\end{equation}
This is weaker than \eqref{polybounded} in the sense that if an
 operator $T : X \rightarrow X$ satisfies \eqref{polybounded} then \eqref{primefuncalc} is bounded. The converse is false even on
   Hilbert space.

The boundedness of $\eqref{primefuncalc}$ is a way to characterize power bounded
 operators on Hilbert space or operators with discrete Gomilko Shi-Feng
  condition on general Banach space. However \eqref{primefuncalc} is not a usual
   bounded functional calculus because we have an estimate of $\norme{\phi'(T)}$ and not
    $\norme{\phi(T)}$. \\

The second part of this article (section 4) is devoted to
       `$\gamma$-versions' of the previous results. By a `$\gamma$-version', we mean
        replacing operator norm boundedness by the
         so-called stronger notion of $\gamma$-boundedness (see \cite{hnvw} and references therein).
           This section starts with some basics about
            $\gamma$-boundedness, $\gamma$-operators and $\gamma'$-operators (see \cite{kal-wei1}). 
            We will introduce a $\gamma$-analogue of the discrete Gomilko Shi-Feng condition for a
             bounded operator $T :X \rightarrow X$ with spectrum included in
              $\overline{\D}$ as follows : there exists $C>0$ such that for any 
            $N \in \N$, for any  $r_1, \ldots , r_N > 1$, and for any
            $x_1, \ldots , x_N \in X$ and $x_1^*, \ldots , x_N^* \in X^*$, we have
            \begin{align}\label{gammagfsintro}	 
            &\sum_{k=1}^{N}\int_{\R}|\langle (r_k+1)(r_k-1) R(r_k e^{it},T)^{2}x_k,x_k^* \rangle| dt \\
            &	\leq C\norme{(x_k)_{k \in \N_N}}_{\gamma(\N_N; X)}\norme{(x^*_k)_{k \in \N_N}}_{\gamma'(\N_N; X^*)} \nonumber.
            \end{align}
            We will show that $T$ satisfies this condition if and only if the set in \eqref{primefuncalc} is $\gamma$-bounded. 
            Moreover, we will show that $T$ is power $\gamma$-bounded, that is the set
              $\{T^n: n\in \N \}$ is $\gamma$-bounded, if and only if $T$ satisfies \eqref{gammagfsintro}.  
               It is important to notice that these results are stated without any assumption on $X$.
                We will conclude this paper with a $\gamma$-version of \cite[Corollary 3.7]{pel1}.\\

         Finally we give a few notation to be used along this paper. We write $\N = \{1,2, \ldots \}$ for the natural numbers and for $N\in \N$ we write $\N_N = \{1,2, \ldots, N\}$ the first N natural numbers. For any Banach spaces $X$ and $Y$, we
          denote by $B(X,Y)$ the algebra of all bounded operators from $X$ into $Y$ equipped with the operator norm, and we set $B(X) := B(X,X)$.
           For $T \in B(X)$ we denote by $\sigma(T)$ the spectrum of
            $T$. For $\lambda \in \C \backslash \sigma(T)$, we put
             $R(\lambda, T) = (\lambda I_X - T)^{-1}$ the resolvent operator. We let $\mathbb{D}$ and $\Dbar$ respectively
             the open and closed unit disk. Also
              the open disk of radius $r$ centered at $0$ and  will be denote
               by $r\mathbb{D}$ while the closed disk of radius $r$ centered
                at $0$ and will be denoted by
                 $r \Dbar$. Also. For any $r\in \R$ we denote
                   by $\Hinf{r}$ the space of all bounded analytic functions
                     $\phi : r\D \rightarrow \C$. This is a Banach space for
                      the norm
         \[
         \norme{\phi}_{\Hinf{r}} := \sup\{|\phi(z)|: z\in r\D \}.
         \]

\section{Discrete Gomilko Shi-Feng condition and derivative functional calculus.}

\subsection{Discrete Gomilko Shi-Feng condition}
	Let $T \in B(X)$ with $\sigma(T) \subset \Dbar$. For $r >1$, $\phi \in \Hinf{r}$ and $m\in \N$, we denote by $\phi^{(m)}(T)$ the operator obtained with Riesz-Dunford calculus. Since $\phi \in \Hinf{r}$ then for each
	  $1<\rho < r$, $\phi^{(m)} \in \Hinf{\rho}$ and therefore the Riesz-Dunford calculus can be applied to the function $\phi^{(m)}$ and the operator $T$. We obtain
	\begin{equation}\label{RDcalculus}
	\phi^{(m)}(T) = \frac{1}{2\pi}\int_{0}^{2\pi}\rho e^{it}\phi^{(m)}(\rho e^{it})R(\rho e^{it},T)dt \in B(X), \quad 1<\rho < r.
	\end{equation}

\begin{lem}
	Let $T \in B(X)$ with $\sigma(T) \subset \Dbar$ and let $m\in \N$. Then for each $r>1$ and $\phi \in \Hinf{r}$, one has
\begin{equation}\label{mIPP}
\phi^{(m)}(T) = \frac{m!}{2\pi}\int_{0}^{2\pi}\rho e^{it}\phi(\rho e^{it})R(\rho e^{it},T)^{m+1}dt, \quad 1<\rho < r.
\end{equation}
\end{lem}
\begin{proof}
	 Let $r>1$, $\phi \in \Hinf{r}$ and $1 < \rho < r$. 
	For each $n \in \N$,
	\[
	\frac{d}{dt}R(\rho e^{it},T)^n = -in\rho e^{it}R(\rho e^{it},T)^{n+1}.
	\]
	Using \eqref{RDcalculus} it follows by $m$ integrations by parts,
	\begin{align*}
	\phi^{(m)}(T) &= \frac{1}{2\pi}\int_{0}^{2\pi}\rho e^{it}\phi^{(m-1)}(\rho e^{it})R(\rho e^{it},T)^2dt\\
	&  = \ldots = \frac{m!}{2\pi}\int_{0}^{2\pi}\rho e^{it}\phi(\rho e^{it})R(\rho e^{it},T)^{m+1}dt.
	\end{align*}
\end{proof}	

We now investigate the discrete Gomilko Shi-Feng condition which appears in \cite{gom2} in the case $m = 1$. The case $m >1$ is inspired by \cite[Proposition 6.3.]{bat-haa}.
\begin{df}\label{dfGFS}
Let $m \geq 1$ be an integer. We say that $T \in B(X)$ has the property $(GFS)_{m}$ if $\sigma(T) \subset \Dbar$ and there exists $C>0$ such that 
	\begin{equation}\label{strinem}
\int_{0}^{2\pi}|\langle R(re^{it},T)^{m+1}x,x^*\rangle |dt \leq \frac{C}{(r+1)(r-1)^m}\norme{x}\norme{x^*},\quad r>1, x\in X, x^* \in X^*.
\end{equation}
\end{df}	

	\begin{prop}\label{gfsimppowbounded}
	Let $T$ with $(GFS)_m$. Then $T$ is power bounded.
	\end{prop}
\begin{proof}
 Let $\phi(z) = \frac{z^{n+m}}{(n+1)\ldots(n+m)}$, then $\phi \in H^{\infty}(r\mathbb{D})$ for any $r>1$ with 
\[\norme{\phi}_{H^{\infty}(r\mathbb{D})} = \frac{r^{n+m}}{(n+1)\ldots(n+m)}
\]
 and furthermore $\phi^{(m)}(z) = z^n$. By \eqref{mIPP}, one has for $r>1$,
\[
T^n = \frac{m!}{2\pi}\int_{0}^{2\pi}re^{it}\phi(re^{it})R(re^{it},T)^{m+1}dt.
\]
It follows, for $x \in X$ and $x^* \in X^*$, 
\begin{align}
|\langle T^nx,x^* \rangle| &\leq \frac{rm!}{2\pi} \norme{\phi}_{\Hinf{r}}\int_{0}^{2\pi}|\langle R(re^{it},T)^{m+1}x,x^*\rangle| dt \nonumber \\
&\label{inepowgfs}\leq   \frac{Cr^{n+m+1}m!}{2\pi(n+1)\ldots(n+m)(r+1)(r-1)^m}\norme{x}\norme{x^*}.
\end{align}
Therefore,

\[
\norme{T^n} \leq \frac{m!Cr^{n+m+1}}{2\pi (n+1)\ldots(n+m)(r-1)^m}.
\]
Letting $r = 1+\frac{1}{n}$ one obtains
\[
\frac{m!Cr^{n+m+1}}{2\pi (n+1)\ldots(n+m)(r-1)^m} = \frac{m!Cn^m(1+\frac{1}{n})^{n}(1+\frac{1}{n})^{m+1}}{2\pi (n+1)\ldots(n+m)}.
\]
But for each $n \in \N$ one has 
\[
\frac{n^m}{(n+1)\ldots(n+m)} \leq 1 \quad \text{and}\quad  \bigg(1+\frac{1}{n}\bigg)^{m+1} \leq 2^{m+1}
\]
and it is well-known that 
\[
\sup_{n\in \N}\Big(1+\frac{1}{n}\Big)^{n} = e.
\]
 Finally, for each $n\in \N$,
\[
\norme{T^n} \leq \frac{2^mm!Ce}{\pi}.
\]
\end{proof}

\begin{rq1}
 We recall that $T \in B(X)$ is a Ritt operator if $\sigma(T) \subset \Dbar$ and there exists $C \geq 0$ such that  \[\forall \lambda \in \C \textbackslash {\Dbar}, \quad \norme{R(\lambda, T)} \leq \displaystyle\frac{C}{|\lambda-1|}.
	\]
Ritt operators play a prominent role in the theory of functional calculus. Indeed they have a specific $H^{\infty}$ functional calculus (see \cite{lem4} and references therein).

	It is easy to check that a Ritt operator has $(GFS)_1$. Indeed let $T$ a Ritt operator and $r>1$, let $x\in X$ and $x^* \in X^*$, one has
\begin{align*}
\int_{0}^{2\pi}|\langle R(re^{it},T)^{2}x,x^*\rangle |dt &\leq \int_{0}^{2\pi} \frac{C}{|re^{it}-1|^2}\norme{x}\norme{x^*} dt \\
&  = \int_{0}^{2\pi} \frac{C}{r^2+1-2rcos(t)}\norme{x}\norme{x^*} dt.
\end{align*}
Now by the residue method (see for example \cite[p.99]{car}) 
\begin{equation}\label{residu}
 \int_{0}^{2\pi} \frac{1}{r^2+1-2rcos(t)}dt = \frac{2\pi}{(r+1)(r-1)},
\end{equation}
therefore one has
\[
\int_{0}^{2\pi}|\langle R(re^{it},T)^{2}x,x^*\rangle |dt  \leq \frac{2\pi C}{(r+1)(r-1)}\norme{x}\norme{x^*}.
\] 	
	
\end{rq1}	
We will need the following stability property. We skip the easy proof.
\begin{prop}\label{GFSsimilarity}
	Let $S \in B(X)$ which has $(GFS)_m$. If $T$ is similar to $S$, that is there exists $U \in B(X)$ invertible such that $T = USU^{-1}$, then $T$ has $(GFS)_m$.   
\end{prop}

We will see that $T$ has $(GFS)_m$ if and only if $T$ has $(GFS)_1$. Let us begin by one implication. For the reverse implication we will use a derivative functional calculus that we introduce in the next sub-section. 

\begin{prop}\label{propGFSmimpGF1}
	Let $m \geq 1$ an integer and $T\in B(X)$. If $T$ has $(GFS)_m$ then $T$ has $(GFS)_k$ for each $1\leq k \leq m$.
\end{prop}

\begin{proof} 
	Let $r>1$ and assume \eqref{strinem}. We show by downward induction that for $1\leq k \leq m$,
	 \begin{equation}\label{inedownind}
	 \int_{0}^{2\pi} |\langle R(re^{it},T)^{k+1}x,x^*\rangle |dt \leq \frac{Cm}{k(r+1)(r-1)^k}\norme{x}\norme{x^*},\quad x\in X, x^* \in X^*.
	 \end{equation}	 
Let $x \in X$, $x^* \in X^*$. Suppose that \eqref{inedownind} is satisfied for $k$ with $1 < k +1 \leq m$. Using the fact that $\frac{d}{dr}R(re^{it},T)^n = -ne^{it}R(re^{it},T)^{n+1}$ for each $n \in \N$ and Fubini's Theorem, one has

	\begin{align}
\int_{0}^{2\pi}\langle R(re^{it},T)^{k}x,x^*\rangle dt &= \int_{0}^{2\pi}\int_{r}^{+\infty}\langle ke^{it} R(ue^{it},T)^{k+1}x,x^*\rangle dudt \nonumber \\
&= \label{eqpropmimp1} k\int_{r}^{+\infty}\int_{0}^{2\pi}e^{it}\langle R(ue^{it},T)^{k+1}x,x^*\rangle dtdu. 
\end{align}
It follows using \eqref{inedownind}, 
	\begin{align*}
	\int_{0}^{2\pi}|\langle R(re^{it},T)^{k}x,x^*\rangle |dt 
	& \leq k\int_{r}^{+\infty}\int_{0}^{2\pi}|\langle R(ue^{it},T)^{k+1}x,x^*\rangle| dtdu \\
	&\leq k\int_{r}^{+\infty}\frac{Cm}{k(u+1)(u-1)^{k}}\norme{x}\norme{x^*}du\\
	&  \leq  \int_{r}^{+\infty}\frac{Cm}{(r+1)(u-1)^{k}}\norme{x}\norme{x^*}du \\
	&= \frac{Cm}{(k-1)(r+1)(r-1)^{k-1}}\norme{x}\norme{x^*}.
	\end{align*}
\end{proof}

\subsection{\dbfc{m}}
We begin by giving a basic result which is a straighforward consequence of Cauchy inequalities. 

\begin{lem}\label{lemcauchyine}
	Let $r > 0$ and $\phi \in H^{\infty}(r\D)$. Then for each $\rho < r$, $\phi^{(m)} \in H^{\infty}(\rho\D)$ with
	\begin{equation}\label{inecau}
	\norme{\phi^{(m)}}_{H^{\infty}(\rho\D)} \leq \frac{m!}{(r-\rho)^m}\norme{\phi}_{H^{\infty}(r\mathbb{D})}.
	\end{equation}
\end{lem}
This Lemma justifies the following definition. 
 \begin{df}
 	Let $T \in B(X)$ with $\sigma(T) \subset \Dbar$ and $m \in \N$. Then $T$ is said to have \textit{ $m$-derivative bounded functional calculus}  if there exists $C>0$ such that 
 	\begin{equation}\label{mfuncalc}
 	\norme{\phi^{(m)}(T)} \leq \frac{C}{(r-1)^m}\norme{\phi}_{H^{\infty}(r\mathbb{D})} \; \text{ for each $\phi \in H^{\infty}(r\mathbb{D})$ \text{ and } $r>1$}.
 	\end{equation}
 \end{df} 

\begin{rq1}
 It is easy to check that \eqref{mfuncalc} is equivalent to:

		\[
	\norme{\phi^{(m)}\big(rT\big)} \leq \frac{C}{(1-r)^m}\norme{\phi}_{H^{\infty}(\mathbb{D})} \; \text{ for each $\phi \in H^{\infty}(\mathbb{D})$ \text{ and } $0<r<1$}.
	\]
\end{rq1}

\begin{thm}\label{th6.4}
	Let $T \in B(X)$ with $\sigma(T) \subset \Dbar$. The following assertions are equivalent for $m \in \N$,
	\begin{enumerate}[label = (\roman*)]
		\item T has $(GFS)_m$;
		\item T has \dbfc{1};
		\item T has \dbfc{m}.		
	\end{enumerate}
\end{thm}
\begin{proof}

		$(i) \Rightarrow (ii)$: According to Proposition \ref{propGFSmimpGF1} $T$ has $(GFS)_1$. Let $r>\rho >1$ and $\phi \in \Hinf{r}$. Then using \eqref{mIPP},
			\begin{align*}
		\norme{\phi'(T)} &= \frac{1}{2\pi}\underset{\norme{x}=1}\sup\underset{\norme{x^*}=1}\sup\Big|\int_{0}^{2\pi}\langle\rho e^{it} \phi(\rho e^{it})R(\rho e^{it},T)^{2}x,x^* \rangle dt\Big| \\
		&\leq \frac{\rho \norme{\phi}_{H^{\infty}(\rho \mathbb{D})}}{2\pi}\underset{\norme{x}=1}\sup\underset{\norme{x^*}=1}\sup\int_{0}^{2\pi}|\langle R(\rho e^{it},T)^{2}x,x^*\rangle|dt\\		
		&\leq \frac{C \rho }{2\pi(\rho -1)(\rho +1)} \norme{\phi}_{H^{\infty}(\rho \mathbb{D})} \leq  \frac{C }{2\pi(\rho -1)} \norme{\phi}_{H^{\infty}(r \mathbb{D})}.
		\end{align*}
		The second to last inequality comes from \eqref{strinem} with $m = 1$. Letting $\rho \rightarrow r$ yields
			\[
						\norme{\phi'(T)}    \leq  \frac{C }{2\pi(r -1)} \norme{\phi}_{H^{\infty}(r \mathbb{D})}.
			\]
		$(ii) \Rightarrow (iii)$: Let $1<\rho <r$ and $\phi \in H^{\infty}(r\mathbb{D})$. According to
		 Lemma \ref{lemcauchyine} $\phi^{(m-1)}\in H^{\infty}(\rho \mathbb{D})$ so applying the
		  \dbfc{1} to $\phi^{(m-1)}$ and again Lemma \ref{lemcauchyine} one has
		\begin{align*}
		\norme{\phi^{(m)}(T)} &= \norme{(\phi^{(m-1)})'(T)} \leq \frac{C }{\rho  -1}\norme{\phi^{(m-1)}}_{H^{\infty}(\rho \mathbb{D})} \\
		&  \leq \frac{(m-1)!C}{(\rho -1)(r-\rho )^{m-1}}\norme{\phi}_{H^{\infty}(r\mathbb{D})}.
		\end{align*}
		Taking $\rho = \frac{r+1}{2}$ yields
		\[
		\norme{\phi^{(m)}(T)}\leq \frac{2^m(m-1)!C}{(r-1)^m}\norme{\phi}_{H^{\infty}(r\mathbb{D})}.
		\]
		$(iii) \Rightarrow (i) $: Let $r>1$, $1 < \rho < r$, $x\in X$ and $x^*\in X^*$. There exists a measurable function $\varepsilon : [0,2\pi] \rightarrow \{z\in \C, \; |z| = 1\}$ such that for each $t \in [0,2\pi)$
		\[
		|\langle R(re^{it},T)^{m+2}x,x^*\rangle | = \varepsilon(t)\langle R(re^{it},T)^{m+2}x,x^*\rangle .
		\]
		Let define $\phi(z) := \displaystyle\frac{1}{(m+1)!}\int_{0}^{2\pi}\frac{\varepsilon(t)}{(re^{it}-z)^2}dt $ for $z\in \rho \mathbb{D}$. Then $\phi \in H^{\infty}( \rho \mathbb{D})$, with 
		\begin{equation}\label{residu2}
		\norme{\phi}_{\Hinf{\rho}} \leq \frac{2\pi}{(m+1)!(r+\rho)(r-\rho)}.
		\end{equation} 
		Indeed for each $\theta \in [0,2\pi)$, using \eqref{residu},
		\begin{align*}
		\int_{0}^{2\pi}\frac{1}{|re^{it}-\rho e^{i \theta }|^2}dt &=  \frac{1}{\rho^2}\int_{0}^{2\pi}\frac{1}{|\frac{r}{\rho}e^{it}-1|^2}dt \\
		& =    \frac{2\pi}{\rho^2(\frac{r}{\rho}+1)(\frac{r}{\rho}-1)} \\
	&= \frac{2\pi}{(r+\rho)(r - \rho)} .
		\end{align*}
		Moreover
		\[
		\phi^{(m)}(z) =\int_{0}^{2\pi}\frac{\varepsilon(t)}{(re^{it}-z)^{m+2}}dt, \; z\in \rho\mathbb{D}. 
		\]
		By Fubini's Theorem one has 
		\[
		\phi^{(m)}(T) = \int_{0}^{2\pi}\varepsilon(t)R(re^{it},T)^{m+2}dt,
		\]
		and finally using \eqref{mfuncalc} and \eqref{residu2},
		\begin{align*}
		\int_{0}^{2\pi}|\langle R(re^{it},T)^{m+2}x,x^*\rangle |& =  \int_{0}^{2\pi} \varepsilon(t)\langle R(re^{it},T)^{m+2}x,x^*\rangle  \\
		&= |\langle \phi^{(m)}(T)x,x^{*}\rangle | \leq \norme{\phi^{(m)}(T)}\norme{x}\norme{x^*} \\
		&\leq\frac{C}{(\rho-1)^m}\norme{\phi}_{H^{\infty}(\rho\mathbb{D})}\norme{x}\norme{x^*} \\
		&\leq \frac{2\pi C}{(m+1)!(\rho-1)^m(r-\rho)(r+\rho)}\norme{x}\norme{x^*} \\
		& \leq \frac{2\pi C}{(m+1)!(r+1)(\rho-1)^m(r-\rho)}\norme{x}\norme{x^*}.
		\end{align*}
Taking $\rho = \frac{r+1}{2}$ one obtains
\[
\int_{0}^{2\pi}|\langle R(re^{it},T)^{m+2}x,x^*\rangle |  \leq \frac{2^{m+2}\pi C}{(m+1)!(r+1)(r-1)^{m+1}}\norme{x}\norme{x^*}.
\]
Thus $T$ has $(GFS)_{m+1}$ and hence by Proposition \ref{propGFSmimpGF1}, $T$ has $(GFS)_{m}$. 
\end{proof}
\begin{cor}
	The condition $(GFS)_m$ does not depend of $m \in \N$. In other words, if $T\in B(X)$ satisfies condition $(GFS)_m$ for one $m \in \N$ then it has $(GFS)_k$ for all $k \in \N$.
\end{cor}
\begin{df}
	Let $T \in B(X)$. If $T$ satisfies one of the three conditions of Theorem \ref{th6.4} then we will say that $T$ is a Gomilko Shi-Feng operator or GFS operator. 
\end{df}
When $X$ is an Hilbert space one obtains the following characterisation.
\begin{cor}\label{corHilcase}
	Let $H$ be a Hilbert space and let $T \in B(H)$ with $\sigma(T) \subset \Dbar$. The following assertions are equivalent,
	\begin{enumerate}[label = (\roman*)]
		\item $T$ is power-bounded;
		\item there is a constant $C>
		 0$ such that for all $r > 1$,
		\begin{align*}
		(r^2-1)\int_{0}^{2\pi}\norme{R(re^{it},T)x}^2dt \leq C\norme{x}^2 \quad (x\in H) \\
		(r^2-1)\int_{0}^{2\pi}\norme{R(re^{it},T)^*y}^2dt \leq C\norme{y}^2 \quad (y\in H); 
		\end{align*}
		\item T is a GFS operator.
        \end{enumerate}
	\end{cor}
\begin{proof}
The implication $(iii) \implies (i)$ is given by Proposition \ref{gfsimppowbounded}. \\
	$(i) \implies (ii)$: Let $r> 1$, one has
	\[
	\int_{0}^{2\pi}\norme{R(re^{it},T)x}^2dt = 	\int_{0}^{2\pi}\norme{\sum_{n=0}^{\infty} \frac{T^nx}{(re^{it})^{n+1}}}^2dt.
	\]
The Fourier-Plancherel Theorem gives
\[
\int_{0}^{2\pi}\norme{\sum_{n=0}^{\infty} \frac{T^nx}{(re^{it})^{n+1}}}^2dt = 2\pi \sum_{n=0}^{\infty} \bigg(\frac{\norme{T^nx}}{r^{n+1}}\bigg)^2.
\]
Then taking $M = \underset{n}\sup\{\norme{T^n} \}$, we have
\[
2\pi \sum_{n=0}^{\infty} \bigg(\frac{\norme{T^n}}{r^{n+1}}\bigg)^2 \leq 2\pi M^2 \sum_{n=0}^{\infty} \frac{1}{r^{2(n+1)}} = \frac{2\pi M^2}{r^2-1},
\]
whence the first inequality in $(ii)$. Since $T^*$ is also power bounded, the second inequality in $(ii)$ holds as well. \\
$(ii) \implies (iii)$: by Cauchy-Schwarz inequality one has 
\begin{align*}
\int_{0}^{2\pi}|\langle R(re^{it},T)^{2}x,x^*\rangle | &= \int_{0}^{2\pi}|\langle R(re^{it},T)x,R(re^{it},T)^*x^*\rangle | \\
& \leq \bigg(\int_{0}^{2\pi}\norme{R(re^{it},T)x}^2dt\bigg)^{\frac{1}{2}}\bigg(\int_{0}^{2\pi}\norme{R(re^{it},T)^*x^*}^2dt\bigg)^{\frac{1}{2}} \\
& \leq \frac{C}{r^2-1}\norme{x}\norme{x^*}.
\end{align*}
If follows that $T$ has $(GFS)_1$, and so $T$ is a GFS operator.	
	\end{proof}
\begin{rqs}\mbox{~} \\
	\begin{enumerate}
	\item We say that $T \in B(X)$ is polynomially bounded when there exists $C>0$ such that for each polynomial $P$ one has
	\[
	\norme{P(T)} \leq C\norme{P}_{\Hinf{}}.
	\] 
	It is easy to show that if $T$ is polynomially bounded, then $T$ is a GFS operator. Indeed let $P$ be a polynomial and $r>1$. Then according to Lemma \ref{lemcauchyine} and the assumption that $T$ is polynomially bounded:
	\[
	\norme{P^{(m)}(T)} \leq C\norme{P^{(m)}}_{\Hinf{}} \leq \frac{Cm!}{(r-1)^m}\norme{P^{(m)}}_{\Hinf{r}}.
	\]
 The converse implication is false. 
 In fact, on any infinite-dimensional Hilbert space there exists a power bounded
	   operator which is not polynomially bounded (see \cite{leb}). But, according to Corollary \ref{corHilcase}, in a Hilbert space an operator
	     is power bounded if and only if it is a GFS operator.

\item When $X$ is not an Hilbert space, the implication $(i) \implies (iii)$ in
 Corollary \eqref{corHilcase} is false. Indeed according to \cite[Theorem 2.2]{gom2},
  if $X$ is a reflexive Banach space and $T \in B(X)$ with $\sigma(T) \subset
   \mathbb{T}$ then $T$ is a scalar type spectral operator if and only if $T$
    and $T^{-1}$ have $(GFS)_1$. Let $U$ denote the shift operator defined
     on $l_p(\Z)$ $(1<p<\infty)$ by
\begin{equation}\label{shiftoperator}
U((x_n)_{n\in \Z}) =(x_{n+1})_{n\in \Z}.
\end{equation}
It is well-known that for $1<p \neq 2 < \infty$, $U$ is not a scalar type spectral operator. It
 follows, that either $U$ or $U^{-1}$ does not have $(GFS)_1$. Further $U$ and $U^{-1}$ are similar. Hence by Proposition \ref{GFSsimilarity}, $U$ does not have $(GFS)_1$. Thus $U$ is a power bounded operator which is not a GFS operator.  
\end{enumerate}
\end{rqs}

\section{Polynomial Besov calculus on general Banach spaces}

We denote by $\B(\D)$ the Banach algebra of all holormorphic functions $f$ on $\D$ such that 
\begin{equation}\label{Bes}
\int_{0}^{1} \underset{t \in [0,2\pi)}\sup|f'(ue^{it})| du < \infty,
\end{equation}
endowed with the norm
\[
\norme{f}_{\B} = \int_{0}^{1} \underset{t \in [0,2\pi)}\sup|f'(ue^{it})| du + \norme{f}_{\Hinf{}}.
\]
We remark that by a change of variables $u =\frac{1}{r}$ one has
 \begin{equation}\label{normeB1overr}
\int_{0}^{1} \underset{t \in [0,2\pi)}\sup|f'(ue^{it})| du = \int_{1}^{\infty} \frac{1}{r^2}\underset{t \in [0,2\pi)}\sup\Big|f'\Big(\frac{e^{it}}{r}\Big)\Big| dr .
 \end{equation}

Let $H$ be a Hilbert space and let $T \in B(H)$ be a power-bounded operator. According to \cite[Theorem 3.8.]{pel1} there exists $C>0$ such that for each polynomial $P$ one has
\[
\norme{P(T)} \leq C \norme{P}_{\B}.
\]
Now since the set of polynomials is dense in $\B(\D)$ the bounded algebra homomorphism $P \rightarrow P(T)$ extends to a bounded algebra homomorphism on $\B(\D)$. 

 We now consider the issue of obtaining a similar bounded Besov calculus on an arbitrary Banach space $X$. Ideas are based on \cite{bgt1}.
 
 \begin{prop}\label{polybesovfc}
 Let $T \in B(X)$ be a $GFS$ operator. Then there exists $C>0$ such that for each polynomial $P$,
 \[
\norme{P(T)} \leq C\norme{P}_{\B}.
 \]
 \end{prop}	
  \begin{proof}
  	Let $n \in \N$ and consider $P_n(z) = z^n$ 
  	Applying \eqref{mIPP} with $m=1$ and $\phi(z) = \frac{z^{n+1}}{n+1}$, one has the following representation, 
  	\[
  	T^n = \frac{1}{2\pi(n+1)}\int_{0}^{2\pi} r^{n+2}e^{i(n+2)t}R^2(re^{it},T)dt,
  	\] 	
for any $r>1$. We deduce that for any $x\in X, x^*\in X^*$,
  \begin{align}
   \langle T^nx,x^* \rangle  &= 2n(n+1)\int_{1}^{\infty} \frac{r^2-1}{r^{2n+3}}dr \langle T^nx,x^* \rangle \nonumber \\
   &  = 2n(n+1)\int_{1}^{\infty} \frac{r^2-1}{r^{2n+3}}\frac{1}{2\pi(n+1)}\int_{0}^{2\pi} r^{n+2} e^{i(n+2)t}\langle R^2(re^{it},T)x,x^*\rangle dtdr \nonumber \\
    	&= \frac{n}{\pi}\int_{1}^{\infty} \frac{r^2-1}{r^{n+1}}\int_{0}^{2\pi}  e^{i(n+2)t}\langle R^2(re^{it},T)x,x^*\rangle dtdr \nonumber \\
    & =   \frac{n}{\pi}\int_{1}^{\infty} (r^2-1)\int_{0}^{2\pi} \frac{e^{3it}}{r^2} \frac{e^{i(n-1)t}}{r^{n-1}} \langle R^2(re^{it},T)x,x^*\rangle dtdr \nonumber \\
     &\label{equB}= \frac{1}{\pi}\int_{1}^{\infty} (r^2-1)\int_{0}^{2\pi} \frac{e^{3it}}{r^2} P_n'\Big(\frac{e^{it}}{r}\Big) \langle R^2(re^{it},T)x,x^*\rangle dtdr,
  \end{align}
  where, in the first equality, we have used the following identity
    	\[
  \int_{1}^{\infty} \frac{r^2-1}{r^{2n+3}}dr = \frac{1}{2n(n+1)}.
  \]
Now let $P$ be a polynomial. Using the linearity of derivative and integration, and previous calculations one obtains
\begin{equation}\label{polybesovrepr}
\langle P(T)x,x^* \rangle = \frac{1}{\pi}\int_{1}^{\infty} (r^2-1)\int_{0}^{2\pi} \frac{e^{3it}}{r^2} P'\Big(\frac{e^{it}}{r}\Big) \langle R^2(re^{it},T)x,x^*\rangle dtdr + P(0) \langle x,x^*\rangle.
\end{equation}
Since $T$ has $(GFS)_1$ there exists $C>0$ such that 
\[
(r^2-1)\int_{0}^{2\pi}|\langle R(re^{it},T)^{2}x,x^*\rangle|dt \leq C\norme{x}\norme{x^*},\quad r>1, x\in X, x^* \in X^*.
\]
Therefore combining this with \eqref{normeB1overr},
\begin{equation}\label{estimationpolybesovfc}
|\langle P(T)x,x^* \rangle| \leq  C'\Big(\int_{1}^{\infty} \frac{1}{r^2} \big|P'\Big(\frac{e^{it}}{r}\Big)\big| dr + |P(0)| \Big)\norme{x}\norme{x^*}\leq C'\norme{P}_{\B}\norme{x}\norme{x^*},
\end{equation}
   which proves the result.
  
 \end{proof}	
 The above proof leads to the following result:
 
 \begin{thm}
 	Let $T\in B(X)$ a $GFS$ operator. Then the mapping $P \rightarrow P(T)$ on the set of polynomials can be extended uniquely to a bounded homomorphism from Besov algebra $\B(\D)$ to the algebra of bounded operators $B(X)$. When $f \in \B(\D)$ let $f(T)$ denote the operator obtained by this extension. 
 	
 	Then for each $f \in \B(\D)$ one has the following representation:
for $x \in X$ and $x^* \in X^*$:
\begin{equation} \label{eqonB}
\langle f(T)x,x^*\rangle = \frac{1}{\pi}\int_{1}^{\infty} (r^2-1)\int_{0}^{2\pi} \frac{e^{3it}}{r^2}f'\Big(\frac{e^{it}}{r}\Big) \langle R^2(re^{it},T)x,x^*\rangle dtdr + \langle f(0)x,x^* \rangle.
\end{equation}
 	
 \end{thm}

\begin{proof}
The first assertion is straighforward since the set of polynomials is dense in Besov algebra $\B(\D)$. Let $f\in \B(\D)$ and denote by $f_{\B}(T)$ the operator defined for $x\in X$ and $x^* \in X^*$ by
\[
\langle f_{\B}(T)x,x^*\rangle = \frac{1}{\pi}\int_{1}^{\infty} (r^2-1)\int_{0}^{2\pi} \frac{e^{3it}}{r^2}f'\Big(\frac{e^{it}}{r}\Big) \langle R^2(re^{it},T)x,x^*\rangle dtdr + \langle f(0)x,x^* \rangle.
\]
The calculation at the end of the proof of Proposition \ref{polybesovfc} shows that this is a well-defined element of $B(X,X^{**})$.
Now let $(P_n)_{n\in \N}$ a sequence of polynomials such that 
\[
\norme{f-P_n}_{\B} \underset{n \rightarrow \infty }\longrightarrow 0.
\]
By definition of $f(T)$, one has for $x\in X$ and $x^* \in X^*$
\[
\langle P_n(T)x,x^*\rangle \underset{n \rightarrow \infty }\longrightarrow \langle f(T)x,x^*\rangle.
\]
Furthermore by the same arguments leading to the estimate \eqref{estimationpolybesovfc} one has for $x\in X$ and $x^* \in X^*$,
\[
|\langle (f_{\B}(T)-P_n(T))x,x^*\rangle|\leq C\norme{f-P_n}_{\B}\norme{x}\norme{x^*}.
\] 
We conclude that $f(T) = f_{\B}(T)$ and therefore we have shown equality \eqref{eqonB}.
\end{proof}	
In general when we extend a functional calculus to a larger class, we do not know what the operators
 obtained on this larger class look like. The above Theorem has more value because we do not only extend
  polynomial calculus to the  Besov algebra but also we are able to give a good representation of
   operators obtained by extension. 

The next result is a converse to Proposition \ref{polybesovfc}.
\begin{prop}
	Let $T \in B(X)$ with $\sigma(T) \subset \Dbar$. If there exists $C>0$ such that for each polynomial $P$
	\[
	\norme{P(T)} \leq C \norme{P}_{\B},
	\]is a GFS operator.
\end{prop}

\begin{proof}
Let $r>1$ and $\phi \in H^{\infty}(r\D)$. Then $\phi' \in \B(\D)$, indeed by \eqref{inecau}
\begin{align*}
\int_{0}^{1} \underset{t \in [0,2\pi)}\sup|\phi''(s e^{it})| ds & = \int_{0}^{1} \norme{\phi''}_{H^{\infty}(s\D)} ds \\
 & \leq  \int_{0}^{1} \frac{2}{(r-s)^2}  \norme{\phi}_{H^{\infty}(r\D)} ds \\
& = \frac{2}{r(r-1)}\norme{\phi}_{H^{\infty}(r\D)}.
\end{align*}
It follows 
\[
\norme{\phi'(T)} \leq C\norme{\phi'}_{\B}\leq \frac{4C}{(r^2-1)} \norme{\phi}_{H^{\infty}(r\D)}.
\]
Therefore $T$ has \dbfc{1}, hence $T$ is a GFS operator.
\end{proof}

\begin{rq1}
It is known (see \cite[Lemma 2.3.7]{white} for details) that there exists $C >0$ such that for each polynomial of degree $N$ one has
\[
\norme{P}_{\B} \leq C log(N+2)\norme{P}_{H^{\infty}(\D)}.
\]	
Therefore if $T$ has $(GFS)_1$ then there exists $C>0$ such that for each polynomial of degree $N$,
\[
\norme{P(T)} \leq C log(N+2)\norme{P}_{H^{\infty}(\D)}.
\]
\end{rq1}
	
\section{Generalizations involving $\gamma$-boundedness}
This main goal of this section is to obtain an extension of Corollary \ref{corHilcase} to Banach spaces. We have noticed that it is not possible to replace $H$ by a Banach space $X$ without additional assumptions. We need assumptions on the operator $T$. We prove an analogue of Corallary \ref{corHilcase} if $\{T^n: n\in \N \}$ is $\gamma$-bounded. We start with some background on the so-called $\gamma$-spaces introduced by Kalton and Weis in \cite{kal-wei1}.

Let $X,Y$ be Banach spaces and 
we let $(\gamma_n)_{n \geq 1}$ be a sequence of independent 
complex valued standard Gaussian variables on some probability space $\Sigma$. We denote by $G(X)$ the closure of 
\[
\big\{\sum_{k=1}^n \gamma_k \otimes x_k\, :\, x_k \in X, \, n\in \N \big\}
\]
in $L^2(\Sigma,X)$. For $x_1,\ldots, x_n \in X$, we let 
\[
\norme{\sum_{k=1}^n \gamma_k \otimes x_k}_{G(X)} := 
\bigg(\underset{\Sigma}\bigint{\norme{\sum_{k=1}^n \gamma_k(\lambda)x_k}^2d\lambda}\bigg)^{\frac{1}{2}}
\] 
denote the induced norm.

\begin{df}
	Let $\mathcal{T} \subset B(X,Y)$ be a set of operators. We say that
	$\mathcal{T}$ is $\gamma$-bounded if
	there exists a constant $C \geq 0 $ such that for all finite sequences 
	$(T_n)_{n=1}^{N} \subset \mathcal{T}$ and $(x_n)_{n=1}^{N} \subset X$, the following inequality holds:
	\begin{equation}\label{Rboundedness}
	\norme{\sum_{n=1}^N \gamma_n\otimes T_nx_n}_{G(Y)} \leq C\norme{\sum_{n=1}^N \gamma_n\otimes x_n}_{G(X)}.
	\end{equation}
	The least admissible constant in the above inequality is called the 
	$\gamma$-bound of  $\mathcal{T}$ and we denote this quantity by $\gamma(\mathcal{T})$. 
	If $\mathcal{T}$ fails to be $\gamma$-bounded, we set $ \gamma(\mathcal{T}) = \infty$. 
\end{df}
The notion of $\gamma$-boundedness is stronger than uniform boundedness, indeed using definition of $\gamma$-boundedness with $N=1$ it is easy to see that $\gamma$-boundedness implies uniform boundedness. 
We will use also two important facts which we sum up in the following proposition.
\begin{prop}\label{propusalgamma}
	Let $\mathcal{T} \subset B(X,Y)$ be  a $\gamma$-bounded set. Then
	\begin{enumerate}
		\item  the closure $\overline{\mathcal{T}}^{so}$  of $\mathcal{T}$ in the 
		strong operator topology is $\gamma$-bounded with 
		$\gamma(\overline{\mathcal{T}}^{so}) = \gamma(\mathcal{T}) $.	
		\item  The absolute convex hull of $\mathcal{T}$, $absconv(\mathcal{T})$ is $\gamma$-bounded with $\gamma(absconv(\mathcal{T})) = \gamma(\mathcal{T})$.
	\end{enumerate}
\end{prop}
We now turn to the definition of $\gamma$-spaces.

Let $H$ be a Hilbert space.
A linear operator $T : H \rightarrow X$ is called $\gamma$-summing if 
\[
\norme{T}_{\gamma} := \sup \norme{\sum_{n=1}^{N}\gamma_n\otimes Th_n }_{G(X)} < \infty,
\]
where the supremum is taken over all finite orthonormal systems $\{h_1,...,h_n \}$ in $H$. 
We let $\gamma_{\infty}(H;X)$ denote the space of all $\gamma$-summing operators and 
we endow it with the norm $\norme{\cdot}_{\gamma}$. Then
$\gamma_{\infty}(H;X)$  is a Banach space. Clearly any finite rank (bounded) operator 
is a $\gamma$-summing operator. We let $\gamma(H;X)$ be
the closure in $\gamma_{\infty}(H;X)$ of the space of finite rank operators from $H$ into $X$. 

We let $\gamma'_{\infty}(H;X)$ be the space of bounded operators $T : H \mapsto X$ such that 
\[
\norme{T}_{\gamma'(H;X)} = \sup\{ trace (T^* \circ S) \; | \; S : H \mapsto X^*,\, \norme{S}_{\gamma(H;X)\leq 1}, \, dim\,S(H) < \infty   \}
\]
and we denote by $\gamma'(H;X)$ the closure of the finite dimensional operators in $\gamma'_{\infty}(H;X)$. See \cite[section 5]{kal-wei1} for details about spaces $\gamma'(H;X)$ and $\gamma_{\infty}'(H;X)$. When $X$ is $K$-convex (see \cite[Section 7.4.]{hnvw} for details about $K$-convexity) then one has 
\[
\gamma'(H;X^*) = \gamma(H;X^*). 
\]

Let $(S,\mu)$ be a measure space. We say that a function 
$f : S \rightarrow X$ is weakly $L^2$ if for each $x^* \in X^*$,  
the function $s \mapsto \langle f(s),x^* \rangle $ is measurable and belongs to $L^2(S)$. 
If $f : S \rightarrow X$ is measurable and weakly $L^2$, one can define an operator $\mathbb{I}_f : L^2(S)\to X$, 
given by
\[
\mathbb{I}_f(g) := \int_S g(s)f(s)d\mu , \quad g\in L^2(S),
\] 
where this integral is defined in the Pettis sense.
We let $\gamma(S;X)$ (resp. $\gamma'(S,X)$) be the space of all measurable and weakly $L^2$
functions $f : S \rightarrow X$ such that $\mathbb{I}_f$
belongs to $\gamma(L^2(S);X)$ (resp. $\gamma'(L^2(S);X)$). We endow it with $\norme{f}_{\gamma(S;X)}:= 
\norme{\mathbb{I}_f}_{\gamma(L^2(S);X)}$ (resp. $\norme{f}_{\gamma'(S;X)}:= 
\norme{\mathbb{I}_f}_{\gamma'(L^2(S);X)}$).

The next result is an inequality of Hölder type \cite[Corollary 5.5.]{kal-wei1}.  
\begin{thm}[$\gamma$-Hölder inequality]\label{inetrace}
	If $f: S \rightarrow X$ and $g : S \rightarrow X^*$ belongs to 
	$\gamma(S;X)$ and $\gamma'(S;X^*)$, respectively, then $\langle f,g\rangle$ belong to $L^1(S)$ and we have
	\[
	\norme{\langle f,g\rangle}_{L^1(S)} \leq \norme{f}_{\gamma(S;X)}\norme{g}_{\gamma'(S;X^*)}.
	\]
\end{thm}
\begin{rq1}\label{rqgamgam'}
Let $N\in \N$, $x_1, \ldots, x_N \in X$ and $x_1^*, \ldots, x_N^* \in X^*$. By construction, 
\begin{equation}\label{GXeqgam}
\norme{\sum_{k=1}^N \gamma_k \otimes x_k }_{G(X)} = \norme{(x_k)_{k \in \N_N}}_{\gamma(\N_N;X)}.
\end{equation}
As a particular case of Theorem \ref{inetrace} one has 
\[
\big|\sum_{k=1}^N\langle x_k, x^*_k \rangle \big| \leq \norme{(x_k)_{k \in \N_N}}_{\gamma(\N_N;X)}\norme{(x^*_k)_{k \in \N_N}}_{\gamma'(\N_N;X^*)}.
\]
Further this inequality is optimal, namely
\begin{equation}\label{dualxk}
 \norme{(x_k)_{k \in \N_N}}_{\gamma(\N_N;X)} = \sup \Big\{ \big|\sum_{k=1}^N\langle x_k, x^*_k \rangle \big| :  (x^*_k)_{k\in \N_N} \subset X^*, \, \norme{(x^*_k)_{k \in \N_N}}_{\gamma'(\N_N;X^*)} \leq 1 \Big\}
\end{equation}
and
\begin{equation}\label{dualxk*}
\norme{(x^*_k)_{k \in \N_N}}_{\gamma'(\N_N;X^*)} = \sup \Big\{ \big|\sum_{k=1}^N\langle z_k,x_k^* \rangle \big| : (x_k)_{k\in \N_N} \subset X, \,  \norme{(x_k)_{k \in \N_N}}_{\gamma(\N_N;X)} \leq 1 \Big\}.
\end{equation}
Furthermore for each sequence $(x^*_k)_{k\in \N} \subset X^*$, one has 
\begin{equation}\label{limitxk*}
\norme{(x^*_k)_{k \in \N}}_{\gamma'(\N;X^*)} = \lim\limits_{N \rightarrow +\infty} \norme{(x^*_k)_{k \in \N_N}}_{\gamma'(\N_N;X^*)}.
\end{equation}

\end{rq1}	
We finally state an extension result (\cite[Theorem 9.6.1.]{hnvw} and \cite[Remark 5.4.]{kal-wei1}).
\begin{lem}\label{isofourierGamma}
Let $H$ and $K$ be Hibert spaces and $U \in B(H^*,K^*)$. Then $U \otimes I_X$ extends uniquely to a bounded linear operator $\widetilde{U} \in B(\gamma(H;X), \gamma(K;X))$ of the same norm.\\
Moreover $U \otimes I_{X^*}$ extends uniquely to a bounded linear operator $\widehat{U} \in B(\gamma'(H;X^*), \gamma'(K;X^*))$ of the same norm.
\end{lem}

\subsection{Power $\gamma$-bounded operators} 
\begin{df}
	Let $T \in B(X)$ be power bounded. We say that $T$ is power $\gamma$-bounded if the set $\{T^n: n\in \N\}$	is $\gamma$-bounded. When $T$ is power $\gamma$-bounded we will denote by $C_{\gamma}$ the $\gamma$-bound $\gamma(\{T^n: n\in \N\cup \{0\} \})$.
\end{df}

\begin{ex}
	The shift operator $U$ defined by \eqref{shiftoperator} is obviously power bounded but for $1 < p\neq 2 < \infty$, $U$ is not power $\gamma$-bounded. Assume that $1<p <2$ and let $(\delta_0)_{n\in \Z }$ the sequence defined by $\delta_0(0)=1$ and 
	$\delta(k) = 0$ for each $k \neq 0$. Then one has 
	\[
	\norme{\bigg(\sum_{i = 1}^n |U^{(i-1)}\delta_0|^2 \bigg)^{\frac{1}{2}}}_{l_p} = n^{\frac{1}{p}} 
	\]
	whereas 
	\[
	\norme{\bigg(\sum_{i = 1}^n |\delta_0|^2 \bigg)^{\frac{1}{2}}}_{l_p} = n^{\frac{1}{2}}. 
	\]
	The proof in the case $2<p<\infty$ is similar. 
\end{ex}

\begin{ex}
	A Ritt operator $T$ is called $\gamma$-Ritt if the set 
	\[
	\{(\lambda-1)R(\lambda,T): \lambda \in \C\backslash\overline{\D} \}
	\]
	is $\gamma$-bounded. It turns out $T$ is $\gamma$-Ritt if and only if the two sets 
	\[
	\{T^n: n\in \N\} \quad \text{and}\quad \{n(I-T)T^n: n\in \N \}
	\]
	are $\gamma$-bounded. It follows that any $\gamma$-Ritt operator is power $\gamma$-bounded. In \cite{arn-lem} we construct a Ritt operator such that the set $\{T^n: n\in \N\}$ is not $\gamma$-bounded (in particular $T$ is not $\gamma$-Ritt). An interesting problem which is still open is to construct a Ritt operator which is not $\gamma$-Ritt but which is power $\gamma$-bounded. 
\end{ex}

\begin{ex}
	Let $(\Omega, \mu)$ be a measure space, $m : \Omega \rightarrow \overline{\D}$ an element of $L^{\infty}(\Omega)$ and $1 \leq p < \infty$. We define the bounded multiplier $T_m \in B(L^p(\Omega))$ by 
	\[
	(T_mf)(s) := m(s)f(s), \quad s \in \Omega.
	\] 
	For each $n \in \N$ one has 
	\[
	(T_m)^n = T_{m^n}.
	\]
	Since for each $n \in \N$, $\normeinf{m^n} \leq 1$, it follows (see \cite[Example 8.1.9]{hnvw}) that $T_m$ is power $\gamma$-bounded. 
\end{ex}

Polynomial boundedness does not imply power $\gamma$-boundedness. Indeed
according to \cite[Proposition 6.6]{lem1} there exist a Banach space $X$ and an invertible operator $T \in B(X)$ such that
$T$ is not power $\gamma$-bounded but there exists a bounded unital homomorphism $\omega : C(\T) \rightarrow B(X)$
such that $\omega(z \mapsto z) = T$. Obviously that operator $T$ is polynomially bounded.

\begin{thm}
	Suppose $X$ has property $(\alpha)$. Let $T \in B(X)$ be invertible and let $m \in \N$. Then $T$ and $T^{-1}$ are GFS operators if and only if $\{T^n, n\in \Z \}$ is $\gamma$-bounded.
\end{thm}
\begin{proof}
	Since $X$ has property $(\alpha)$ then according to \cite[Proposition 6.3. (2)]{lem1}, 
	the set $\{T^n, n\in \Z \}$ is $\gamma$-bounded if and only if there is a bounded unital homomorphism $\omega : C(\T) \rightarrow B(X)$
	such that $\omega(z \mapsto z) = T$. 
	Furthermore, according to the proof of \cite[Theorem 2.2.]{gom2}, there is a bounded unital homomorphism $\omega : C(\T) \rightarrow B(X)$
	such that $\omega(z \mapsto z) = T$ if and only if $T$ and $T^{-1}$ have $(GFS)_1$. The result is now straighforward.
\end{proof}


\subsection{$\gamma$-discrete Gomilko Shi-Feng condition}

\begin{df}
	
	Let $m\geq 1$ be an integer and $T \in B(X)$ with $\sigma(T) \subset \Dbar$.
	We say that $T$ has property $(\gamma$-$GFS)_{m}$ if there exists
	a constant $C>0$ such that for any 
	$N \in \N$, for any  $r_1, \ldots , r_N > 1$, and for any
	$x_1, \ldots , x_N \in X$ and $x_1^*, \ldots , x_N^* \in X^*$, we have
	\begin{align}\label{gammagfs}	 
	&\sum_{k=1}^{N}\int_{\R}|\langle (r_k+1)(r_k-1)^m R(r_k e^{it},T)^{m+1}x_k,x_k^* \rangle| dt \\
	&	\leq C\norme{(x_k)_{k \in \N_N}}_{\gamma(\N_N; X)}\norme{(x^*_k)_{k \in \N_N}}_{\gamma'(\N_N; X^*)} \nonumber.
	\end{align}
\end{df}
Obviously if $T$ has $(\gamma$-$GFS)_m$ then $T$ has $(GFS)_m$. Moreover on Hilbert spaces, properties $(\gamma$-$GFS)_m$ and $(GFS)_m$ are equivalent.

\begin{prop}\label{gamgfsimplgambound}
	If $T$ has $(\gamma$-$GFS)_{m}$ then $T$ is power $\gamma$-bounded, that is, the set $\{T^n: n\in \N \}$ is $\gamma$-bounded. 
\end{prop}	

\begin{proof}
	The proof is similar to the proof of Proposition \ref{gfsimppowbounded}. Let $N \in \N$, $n_1,\cdots, n_N \in \N$, $r_1,\ldots,r_N > 1$, $x_1,\ldots,x_N \in X$ and $x_1^*,\ldots,x_N^* \in X^*$. For $x_1^*,\ldots,x_N^* \in X^*$ such that $\norme{(x^*_k)_{k \in \N_N}}_{\gamma'(\N_N; X^*)} \leq 1$ the proof of \eqref{inepowgfs} and the assumption yield
	\begin{align*}
	&\sum_{k=1}^N \frac{2\pi(n_k+1)\ldots(n_k+m)(r_k+1)(r_k-1)^m}{r^{n_k+m+1}m!}|<T^{n_k}x_k,x_k^*>|\\
	& \leq \sum_{k=1}^N \int_{0}^{2\pi}(r_k+1)(r_k-1)^m|\langle R(r_ke^{it},T)^{m+1}x_k,x_k^* \rangle| dt \\
	& \leq C\norme{(x_k)_{k \in \N_N}}_{\gamma(\N_N; X)}\norme{(x^*_k)_{k \in \N_N}}_{\gamma'(\N_N; X^*)} \\
	& \leq C\norme{(x_k)_{k \in \N_N}}_{\gamma(\N_N; X)}.
	\end{align*}
	Taking $r_k = 1 + \frac{1}{n_k}$ one obtains
	\[
	\sum_{k=1}^N |<T^{n_k}x_k,x_k^*>| \leq \frac{2^mm!Ce}{\pi}\norme{(x_k)_{k \in \N_N}}_{\gamma(\N_N; X)}.
	\]
	Taking the supremum over $x_1^*, \ldots, x_N^*$, such that $\norme{(x^*_k)_{k \in \N_N}}_{\gamma'(\N_N; X^*)} \leq 1$, one obtains according to \eqref{dualxk}
	\[
	\norme{\big(T^{n_k}x_k\big)_{k \in \N_N}}_{\gamma(\N_N;X)} \leq \frac{2^mm!Ce}{\pi }\norme{(x_k)_{k \in \N_N}}_{\gamma(\N_N; X)}.
	\]
	Hence the set $\{T^n: n\in \N \}$ is $\gamma$-bounded.
\end{proof}
\begin{lem}\label{equigamGFS}
	Let $m\in \N$ and $T \in B(X)$ with $\sigma(T) \subset \Dbar$. Then $T$ has property $(\gamma\text{-}GFS)_m$  if and only if the set 
	\begin{equation}\label{gammaSet}
	\Big\{(r+1)(r-1)^m\displaystyle\int_{0}^{2\pi}\varepsilon(t)R(re^{it},T)^{m+1}dt \, \big|\,
	\varepsilon : [0,2\pi)\rightarrow \overline{\mathbb{D}}\text{ measurable},\, r> 1 \Big\}
	\end{equation}
	is $\gamma$-bounded.
\end{lem}
\begin{proof}
	Let $T \in B(X)$ with $(\gamma$-$GFS)_m$. Let $N \in \N$, $r_1, \ldots , r_N > 1$, 
	$x_1, \ldots , x_N \in X$ and $\varepsilon_k : [0,2\pi)\rightarrow \overline{\mathbb{D}}$ measurable. For $x_1^*, \ldots , x_N^* \in X^*$ with $\norme{(x_k^*)_{k \in \N_N}}_{\gamma'(\N_N; X^*)}\leq 1$ one has, using \eqref{gammagfs}, 
	\begin{align*}
	\Big|\sum_{k=1}^N\big\langle \Big(&(r_k+1)(r_k-1)^m\displaystyle\int_{0}^{2\pi}\varepsilon_k(t)R(r_ke^{it},T)^{m+1}dt\Big)x_k,x^*_k 	\big\rangle \Big| \\
	&\leq  C\norme{(x_k)_{k \in \N_N}}_{\gamma(\N_N; X)}. \nonumber
	\end{align*}
	Now take the supremum over $x_1^*, \ldots , x_N^* \in X^*$ one obtains by \eqref{dualxk}
	\begin{equation*}
	\norme{\Big((r_k+1)(r_k-1)^m\displaystyle\int_{0}^{2\pi}\varepsilon_k(t)R(r_ke^{it},T)^{m+1}dtx_k\Big)_{k \in \N_N}}_{\gamma(\N_N; X)} 
	\leq  C\norme{(x_k)_{k \in \N_N}}_{\gamma(\N_N; X)}. 
	\end{equation*}
	This means that the set in \eqref{gammaSet}	is $\gamma$-bounded.
	
	Suppose that the set in  \eqref{gammaSet}	is $\gamma$-bounded.  Let $N \in \N$, $r_1, \ldots , r_N > 1$, 
	$x_1, \ldots , x_N \in X$ and $x_1^*, \ldots , x_N^* \in X^*$. For $k \in \N_N$, let $\varepsilon_k : [0,2\pi) \rightarrow \overline{\D}$ measurable such that 
	\begin{align*}
	\int_{\R}|&\langle (r_k+1)(r_k-1)^m R(r_k e^{it},T)^{m+1}x_k,x_k^* \rangle| dt  \\
	&= \big\langle \big((r_k+1)(r_k-1)^m \int_{\R} \varepsilon_k(t)  R(r_k e^{it},T)^{m+1}dt\big )x_k,x_k^* \big\rangle .
	\end{align*}
	Then by Theorem \ref{inetrace} and the assumption,
	\begin{align*}
	\big\langle &\Big((r_k+1)(r_k-1)^m \int_{\R} \varepsilon_k(t)  R(r_k e^{it},T)^{m+1}dt\big )x_k,x_k^* \big\rangle  \\ &\leq \norme{\Big(\sum_{k=1}^N(r_k+1)(r_k-1)^m\displaystyle\int_{0}^{2\pi}\varepsilon_k(t)R(r_ke^{it},T)^{m+1}dtx_k\Big)_{k \in \N_N}}_{\gamma(\N_N; X)} \norme{(x^*_k)_{k \in \N_N}}_{\gamma'(\N_N; X^*)}.
	\\& \leq C \norme{(x_k)_{k \in \N_N}}_{\gamma(\N_N; X)}\norme{(x^*_k)_{k \in \N_N}}_{\gamma'(\N_N; X^*)}.
	\end{align*}
\end{proof}

\begin{prop}\label{propPm}
	Let $T \in B(X)$ with $(GFS)_m$. If $T$ has $(\gamma$-$GFS)_m$ then $T$ has $(\gamma$-$GFS)_p$ for $ 1 \leq p \leq m$. 
\end{prop}	

\begin{proof}
	We proceed by induction, showing that if $m\geq 2$, then  $(\gamma\text{-}GFS)_m$
	implies \\ $(\gamma\text{-}GFS)_{m-1}$.

	Suppose that $T$ has property $(\gamma\text{-}GFS)_{m}$, with $m\geq 2$. 
	Let $\varepsilon :  [0,2\pi) \rightarrow \overline{\mathbb{D}}$ be a measurable function
	and let $r > 1$. Then by the argument showing \eqref{eqpropmimp1} and a change of variable one has 
	\begin{align*}
	&(r+1)(r-1)^{m-1}\int_{0}^{2\pi} \varepsilon(t)R(re^{it},T)^{m} dt \\ & = r(m-1)\int_{1}^{+\infty}(r+1)(r-1)^{m-1}\int_{0}^{2\pi}\varepsilon(t) R(rue^{it},T)^{m+1} dtdu \\
	&= r(m-1)\int_{1}^{+\infty}\frac{(r+1)(r-1)^{m-1}}{(ru+1)(ru-1)^m}\int_{0}^{2\pi}(ru+1)(ru-1)^m\varepsilon(t) R(rue^{it},T)^{m+1} dtdu.
	\end{align*}
	
	Let $\mathcal{T} $ be the set (\ref{gammaSet}). By assumption and Lemma \ref{equigamGFS}, $\mathcal{T} $ is $\gamma$-bounded hence by
 \cite[Theorem 8.5.2.]{hnvw},
	the set
	\begin{align*}
	\Gamma := \big\{ \int_{1}^{\infty}& \phi(u)\int_{0}^{2\pi} (ru+1)(ru-1)^m 
	\varepsilon(t)R(	rue^{it},T)^{m+1} dt du\, \big|\, \\
	& \varepsilon : [0,2\pi) \rightarrow \overline{\mathbb{D}}\text{ measurable},\, r>  1,\,
	\phi\in L^1((1, \infty)),\, \norme{\phi}_{L^1} \leq 1 \big\}
	\end{align*}
	is $\gamma$-bounded. Since 
	\[
	\norme{r(m-1)\frac{(r+1)(r-1)^{m-1}}{(r\cdot +1)(r\cdot - 1)^{m}}}_{L^1((1,\infty))} \leq 1,
	\]
	the above calculation shows that the set
	\[
	\big\{(r+1)(r-1)^{m-1}\displaystyle\int_{0}^{2\pi}\varepsilon(t)R(re^{it},T)^{m}dt
	\, :\, \varepsilon : \R \rightarrow \overline{\mathbb{D}}\text{ measurable},\, r>1 \big\}  
	\]
	is included in $\Gamma$, hence is $\gamma$-bounded. Hence, by Lemma \ref{equigamGFS}, 
	the operator $T$ has property $(\gamma\text{-}GFS)_{m-1}$. 
\end{proof}

\subsection{\gdbfc{m}}

\begin{df}
	Let $T \in B(X)$ with $\sigma(T) \subset \Dbar$. Then $T$ is said to have \gdbfc{m} if the set 
	\[
	\big\{(r-1)^m\phi^{(m)}(T): r>1,\, \phi \in \Hinf{r},\, \norme{\phi}_{H^{\infty}(r\mathbb{D})} \leq 1 \big\}
	\]
	is $\gamma$-bounded.
\end{df}

We give now the $\gamma$-bounded version of Theorem \ref{th6.4}.
\begin{thm}\label{gamth6.4}
	Let $T \in B(X)$ with $\sigma(T) \subset \Dbar$. The following assertions are equivalent for $m \in \N$,
	\begin{enumerate}[label = (\roman*)]
		\item T has $(\gamma$-$GFS)_m$;
		\item T has \gdbfc{1};
		\item T has \gdbfc{m}.
	\end{enumerate}
\end{thm}

\begin{proof}
	$(i) \Rightarrow (ii)$: First, by Proposition \ref{propPm}, $T$ has $(\gamma\text{-}GFS)_{1}$.	Let $N \in \N$ and $r_1,\ldots,r_N > 1$, $\phi_1 \in \Hinf{r_1},\ldots,\phi_N\in \Hinf{r_N}$ with $\norme{\phi_k}_{\Hinf{r_k}}\leq 1$. Let 
	$x_1,\ldots,x_N \in X$ and $x_1^*,\ldots,x_N^* \in X^*$. Using \eqref{mIPP} and arguing as in the proof of $(i) \implies (ii)$ in Theorem \ref{th6.4}, one has
	\[
	\sum_{k=1}^N |\langle(r_k-1)\phi'_k(T)x_k,x^*_k\rangle| \leq   \frac{1}{2\pi}\sum_{k=1}^N \norme{\phi_k}_{\Hinf{r_k}}
	\int_{0}^{2\pi}\big|\langle(r_k^2 -1 )R(r_ke^{it},T)^{2}x_k,x^*_k\rangle\big|dt.  
	\]
	Suppose $\norme{(x_k^*)_{k \in \N_N}}_{\gamma'(\N_N; X^*)}\leq 1$. Since $\norme{\phi_k}_{\Hinf{r_k}} \leq 1$ and $T$ has $(\gamma\text{-}GFS)_{1}$, one obtains 
		\[
	\sum_{k=1}^N |\langle(r_k-1)\phi'_k(T)x_k,x^*_k\rangle| \leq  \frac{C}{2\pi}\norme{(x_k)_{k \in \N_N}}_{\gamma(\N_N; X)}.
	\]
	Taking the supremum over $x_1^*,\ldots ,x_N^*$ this implies, thanks to \eqref{dualxk}, that $T$ has \gdbfc{1}.
	

	\smallskip
	$(ii) \Rightarrow (iii)$: It follows from the assumption that the set 
	\[
	\Delta := \Big\{(\rho-1)\phi^{(m)}(T): 1<\rho<r , \,  \phi \in \Hinf{r}, \; \norme{\phi^{(m-1)} }_{\Hinf{\rho}} \leq 1 \Big\}
	\]
	is $\gamma$-bounded. For any $1 <\rho < r$ and $\phi \in \Hinf{r}$
	with $\norme{\phi}_{\Hinf{r}} \leq 1$, by Lemma \ref{lemcauchyine}  one has
	\[
	\norme{\frac{(r-\rho)^{m-1}}{(m-1)!}\phi^{(m-1)}}_{\Hinf{\rho}} \leq 1.
\]
Hence
	\[
	\bigg\{\frac{(\rho - 1)
		(r - \rho)^{m-1}}{(m-1)!}\phi^{(m)}(T): 1<\rho<r, \,\phi \in  \Hinf{r}, \;
	\norme{\phi}_{\Hinf{r}}\leq 1 \bigg\} \subset \Delta. 
	\]
	Taking $\rho = \frac{r+1}{2}$ in the above set, we obtain 
	\[
\bigg\{\frac{(r-1)^m}{2^m (m-1)!}\phi^{(m)}(T)
 : r>1, \, \phi \in  \Hinf{r}, \;
	\norme{\phi}_{\Hinf{r}}\leq 1 \bigg\} \subset \Delta.
	\]
	Hence the above set is $\gamma$-bounded. Thus $T$ has \gdbfc{m}.

	
	\smallskip
	$(iii) \Rightarrow (i)$: Let $r_1,\ldots,r_N> 1$, $1<\rho_1 < r_1, \ldots, 1<\rho_N < r_N$,
	$x_1, \ldots, x_N \in X$ and  $x^*_1,\ldots, x^*_N \in X^{*}$. Let us introduce measurable functions
	$\varepsilon_1, ..., \varepsilon_N : [0,2\pi) \to \overline{\mathbb{D}}$ such that  
			\[
	|\langle R((r_k)e^{it},T)^{m+2}x_k,x_k^*\rangle | = \varepsilon_k(t)\langle R((r_k)e^{it},T)^{m+2}x_k,x_k^*\rangle .
	\]
	for all $k=1,\ldots, N$ and all $t \in [0,2\pi)$. Next we set
	\[
	\phi_{k}(z) := \frac{(r_k+\rho_k)(r_k-\rho_k)}{2\pi}\int_{0}^{2\pi}\frac{\varepsilon_k(t)}{(r_ke^{it}-z)^2}dt, \quad z\in r_k\D.
	\]
	The computations in the proof of Theorem \ref{th6.4} show that $\norme{\phi_k}_{\Hinf{\rho_k}} \leq 1$ and
	\[
	\phi_k^{(m)}(T) =\frac{(m+1)!(r_k+\rho_k)(r_k-\rho_k)}{2\pi}\int_{0}^{2\pi}\varepsilon_k(t)R(r_ke^{it},T)^{m+2}dt.
	\]
	Therefore it follows from $(iii)$ that we have the estimate
	\begin{align*}
	\sum_{k=1}^{N} & \frac{(m+1)!(\rho_k -1)^m(r_k-\rho_k)(r_k+1)}{2\pi}\int_{0}^{2\pi}|\langle  R(r_ke^{it},T)^{m+2}x_k,x^*_k \rangle|dt \\
	& \leq \sum_{k=1}^{N}   \frac{(m+1)!(\rho_k -1)^m(r_k-\rho_k)(r_k+\rho_k)}{2\pi}\int_{0}^{2\pi}\varepsilon_k(t)\langle R(r_ke^{it},T)^{m+2}x_k,x^*_k \rangle dt  \\
	&= \sum_{k=1}^{N}  \big|\langle (\rho_k -1)^m\phi_{k}^{(m)}(T) x_k,x^*_k \rangle\big|dt \\
	& \leq C\norme{(x_k)_{k \in \N_N}}_{\gamma(\N_N; X)}\norme{(x^*_k)_{k \in \N_N}}_{\gamma'(\N_N; X^*)}.
	\end{align*}
	The last inequality comes from Theorem \ref{inetrace}. Now we choose $\rho_k =  \frac{r_k + 1}{2}$ in the above estimate.
	We obtain the following inequality
	\begin{align*}
	&\sum_{k=1}^{N} \int_{\R}|\langle (r_k+1)(r_k-1)^{m+1}
	R(r_ke^{it},T)^{m+2}x_k,x^*_k \rangle|dt \\
	&\leq \frac{2^{m+2}\pi C}{(m+1)!}\norme{(x_k)_{k \in \N_N}}_{\gamma(\N_N; X)}
	\norme{(x^*_k)_{k \in \N_N}}_{\gamma'(\N_N; X^*)}.
	\end{align*}	
	This shows that $T$ has $(\gamma\text{-}GFS)_{m+1}$.
	Then by Proposition \ref{propPm}, $T$ has $(\gamma\text{-}GFS)_{m}$. 
\end{proof}
\begin{df}
	Let $T \in B(X)$. If $T$ satisfies one of the three conditions of Theorem \ref{gamth6.4} then we will say that $T$ is a $\gamma$-GFS operator. 
\end{df}

\subsection{Characterization of power $\gamma$-bounded on Banach space $X$}

In the following, the space $L^2((0,2\pi))$ will be equipped with the norm
\[
\norme{f}^2_2 = \frac{1}{2\pi}\int_{0}^{2\pi}|f(s)|^2ds, \quad f \in  L^2((0,2\pi)).
\]
Therefore, the Fourier-Parseval operator 
\[\begin{array}{ccccc}
\mathcal{F} & : & L^2([0,2\pi)) & \to & l^2_\Z \\
& & f & \mapsto & (c_n(f))_{n\in \Z}, \\
\end{array}\] is an isometry. Here $c_n(f)$ is the $n$-th Fourier coefficient defined by 
\[
c_n(f) = \frac{1}{2\pi} \int_{0}^{2\pi}f(t)e^{-int}dt.
\]	
We give a characterization of power $\gamma$-bounded operators.
\begin{thm}\label{shifengRdisthm}
	Let $X$ be a Banach space. Following assertions are equivalent : 
	\begin{enumerate}[label = (\roman*)]
		\item The operator $T$ is power $\gamma$-bounded;
		\item The spectrum set $\sigma(T)$ is included in $\Dbar$ and there exists a constant $C>0$
		such that for all $N\in \N$, $r_1,\ldots, r_N >1$, $x_1,\ldots,x_N \in X$ and
		$x^*_1,\ldots,x^*_N \in X^*$ the functions $(t,k) \mapsto \sqrt{r_k^2-1}R(r_ke^{it}, T)x_k$
		and  $(t,k) \mapsto \sqrt{r_k^2-1}R(r_ke^{it}, T^*)x^*_k$ are in
		$\gamma([0,2\pi)\times\N_N;X)$ and $\gamma'([0,2\pi)\times\N_N;X^*)$ respectively, and
		satisfy
		\begin{equation}\label{ine1dis}
		\norme{(t,k) \mapsto \sqrt{r_k^2-1}R(r_ke^{it}, T)x_k}_{\gamma([0,2\pi)\times\N_N;X)} \leq C\norme{(x_k)_{k \in \N_N}}_{\gamma(\N_N; X)}
		\end{equation}
		and 
		\begin{equation}\label{ine1dis*}
		\norme{(t,k) \mapsto \sqrt{r_k^2-1}R(r_ke^{it}, T^*)x^*_k}_{\gamma'([0,2\pi)\times\N_N;X^*)} \leq \norme{(x_k^*)_{k \in \N_N}}_{\gamma'(\N_N; X^*)};
		\end{equation}
		\item $T$ is a $\gamma$-GFS operator.
	\end{enumerate}
\end{thm}

\begin{proof}
$(iii)\implies(i)$ is Proposition \ref{gamgfsimplgambound}.

	$(i) \implies (ii)$ : Recall that for each $k \in \N_N$ one has
	\[
	R(r_ke^{it},T)x_k =  \sum_{n=0}^{\infty} \frac{T^{n}x_k}{(r_ke^{it})^{n+1}} =  \sum_{n=0}^{\infty}  \frac{T^{n}x_k}{r_k^{n+1}}e^{-i(n+1)t}.
	\]
	We apply Lemma \ref{isofourierGamma} with $H = l^2_{N}\overset{2}{\otimes} L^2((0,2\pi))$, $K =  l^2_{N}\overset{2}{\otimes} l^2_{\Z}$ and $U : H \rightarrow K$ defined by
	\[
	U\big(\sum_{k=1}^N e_k \otimes f_k\big) = \sum_{k=1}^N e_k \otimes (c_n (f_k))_{n\in \Z}, \quad f_1,\ldots,f_N \in L^2((0,2\pi)). 
	\]
We obtain that $(n,k) \mapsto \sqrt{r_k^2-1}\frac{T^{n}x_k}{r_k^{n+1}}$ belongs to $\gamma(\N\cup\{0\}\times\N_N;X)$ and
		\[
		\norme{(t,k) \mapsto \sqrt{r_k^2-1}R(r_ke^{it}, T)x_k}_{\gamma([0,2\pi)\times\N_N;X)}   =	\norme{(n,k) \mapsto  \sqrt{r_k^2-1}\frac{T^{n}x_k}{r_k^{n+1}}}_{\gamma(\N\cup\{0\}\times\N_N;X)}.
	\]
But, one has by \eqref{GXeqgam}
\[
	\norme{(n,k) \mapsto \sqrt{r_k^2-1} \frac{T^{n}x_k}{r_k^{n+1}}}_{\gamma(\N\cup\{0\}\times\N_N;X)} 
	= \norme{\sum_{k=1}^{N} \sqrt{r_k^2-1}\sum_{n=0}^{\infty}\frac{1}{r_k^{n+1}}\gamma_{k,n} \otimes   T^nx_k}_{G(X)},
\]
where $(\gamma_{k,n})_{k\in \N_N, n\in \N\cup\{0\}}$ is a family of independant complex valued standard Gaussian variables.
Using the $\gamma$-boundedness of $\{T^n: n\in \N\}$ it follows
\[
\norme{\sum_{k=1}^{N} \sqrt{r_k^2-1}\sum_{n=0}^{\infty}\frac{1}{r_k^{n+1}}\gamma_{k,n} \otimes   T^nx_k}_{G(X)}\leq C_{\gamma}\norme{\sum_{k=1}^{N} \sqrt{r_k^2-1}\sum_{n=0}^{\infty}\frac{1}{r_k^{n+1}}\gamma_{k,n} \otimes   x_k}_{G(X)}. 
\]
Furthermore $\norme{\Bigg(\sqrt{r_k^2-1}\frac{1}{r_k^{n+1}}\Bigg)_{n \in \N \cup \{0\}}}_{l^2} =  1$ for each $k \in \N_N$ therefore 
\[\Big(\sum_{n=0}^{\infty}\sqrt{r_k^2-1}\frac{1}{r_k^{n+1}}\gamma_{k,n}\Big)_{k\in \N_N} 
\]
is a sequence of independent standard Gaussian variables, which implies
\[
\norme{\sum_{k=1}^{N} \sqrt{r_k^2-1}\sum_{n=0}^{\infty}\frac{1}{r_k^{n+1}}\gamma_{k,n} \otimes   x_k}_{G(X)} = \norme{(x_k)_{k\in \N_N}}_{\gamma(\N_N;X)}.
\]
Hence one obtains \eqref{ine1dis}.

Likewise, using Lemma \ref{isofourierGamma} we obtain $(n,k) \mapsto \sqrt{r_k^2-1}\frac{{T^*}^{n}x_k}{r_k^{n+1}}$ belongs to $\gamma'(\N\cup\{0\}\times\N_N;X^*)$ and
\[
\norme{(t,k) \mapsto \sqrt{r_k^2-1}R(r_ke^{it}, T^*)x^*_k}_{\gamma'([0,2\pi)\times\N_N;X^*)}   =	\norme{(n,k) \mapsto  \sqrt{r_k^2-1}\frac{{T^*}^{n}x^*_k}{r_k^{n+1}}}_{\gamma'(\N\cup\{0\}\times\N_N;X^*)}.
\] 
Furthermore, one has by \eqref{limitxk*}
\[
\norme{(n,k) \mapsto  \sqrt{r_k^2-1}\frac{{T^*}^{n}x^*_k}{r_k^{n+1}}}_{\gamma'(\N\cup\{0\}\times\N_N;X^*)} = \lim\limits_{M \rightarrow \infty}\norme{(n,k) \mapsto  \sqrt{r_k^2-1}\frac{{T^*}^{n}x^*_k}{r_k^{n+1}}}_{\gamma'(\N_M\cup\{0\}\times\N_N;X^*)}.
\] 
But, by \eqref{dualxk*}
\begin{align} \label{proofthmgamgfs}
&\norme{(n,k) \mapsto  \sqrt{r_k^2-1}\frac{{T^*}^{n}x^*_k}{r_k^{n+1}}}_{\gamma'(\N_M\cup\{0\}\times\N_N;X^*)} \\
&= \sup \Big\{ \Big|\sum_{n = 1 }^M \sum_{k=1}^N \big\langle x_{n,k}, \sqrt{r_k^2-1}\frac{{T^*}^{n}x^*_k}{r_k^{n+1}}\big\rangle \Big| : \norme{(x_{n,k})_{n \in \N_M\cup \{0\}, k \in \N_N}}_{\gamma(\N_M\cup \{0\} \times \N_N,X)}\leq 1 \Big\} \nonumber,
\end{align}
and for $M\in \N$, using Theorem \ref{inetrace},
\begin{align*}
\Big|\sum_{n = 1 }^M \sum_{k=1}^N \langle x_{n,k}, \sqrt{r_k^2-1}\frac{{T^*}^{n}x^*_{k}}{r_k^{n+1}}\rangle \Big| & = \Big|\sum_{n = 1 }^M \sum_{k=1}^N \langle \sqrt{r_k^2-1}\frac{{T}^{n}x_{n,k}}{r_k^{n+1}}, x_k^*\rangle \Big| \\
& \leq \norme{\sqrt{r_k^2-1}\frac{{T}^{n}x_{n,k}}{r_k^{n+1}}}_{{\gamma(\N_M\cup \{0\} \times \N_N,X)}}\norme{(x_k^*)_{k \in \N_N}}_{\gamma'(\N_N; X^*)}.
\end{align*}
By previous computations which leads to \eqref{ine1dis}, for each $M \in \N$ one has
\[
\norme{\sqrt{r_k^2-1}\frac{{T}^{n}x_{n,k}}{r_k^{n+1}}}_{\gamma(\N_M\cup \{0\}\times \N_N;X)} \leq C_{\gamma}\norme{(x_k)_{k \in \N_N}}_{\gamma(\N_N; X)}.
\] 
This implies 
\[
\norme{(n,k) \mapsto  \sqrt{r_k^2-1}\frac{{T^*}^{n}x^*_k}{r_k^{n+1}}}_{\gamma'(\N_M\cup\{0\}\times\N_N;X^*)} \leq C_{\gamma}\norme{(x_k^*)_{k \in \N_N}}_{\gamma'(\N_N; X^*)}.
\]
Hence one obtains \eqref{ine1dis*}. \\
$(ii) \implies (iii)$ : Let $N\in \N$, $r_1,\ldots, r_N >1$, $x_1,\ldots,x_N \in X$ and $x^*_1,\ldots,x^*_N \in X^*$. One has
\begin{align*}
\sum_{k=1}^{N}&\int_{\R}|\langle (r_k^2-1) R(r_k e^{it},T)^{2}x_k,x_k^* \rangle| dt \\
& = \sum_{k=1}^{N}\int_{\R}|\langle (r_k^2-1) R(r_k e^{it},T)x_k,R(r_k e^{it},T)^*x_k^* \rangle|dt\\
& = \norme{(t,k) \mapsto  \langle \sqrt{r_k^2-1}R(r_k e^{it},T)x_k,\sqrt{r_k-1}R(r_k e^{it},T)^*x_k^* \rangle}_{L^1((0,2\pi)\times \N_N)}. 
\end{align*}
By Theorem \ref{inetrace} and assumptions \eqref{ine1dis} and \eqref{ine1dis*},
\begin{align*}
&\norme{(t,k) \mapsto   \langle \sqrt{r_k^2-1}R(r_k e^{it},T)x_k,\sqrt{r_k^2-1}R(r_k e^{it},T)^*x_k^* \rangle}_{L^1((0,2\pi)\times \N_N)}  \\
 & \leq \norme{\sqrt{r_k^2-1} R(r_k e^{it},T)x_k}_{\gamma((0,2\pi)\times \N_N;X)}\norme{\sqrt{r_k^2-1}R(r_k e^{it},T)^*x_k^*}_{\gamma'((0,2\pi)\times \N_N;X^*)} \\
& \leq C^2\norme{(x_k)_{k \in \N_N}}_{\gamma(\N_N; X)}\norme{(x_k)_{k \in \N_N}}_{\gamma'(\N_N; X^*)}.
\end{align*}
Hence $T$ has $(\gamma$-$GFS)_1$ and therefore it is a $\gamma$-GFS operator.
\end{proof}

\begin{rq1}
When $X$ is $K$-convex \eqref{ine1dis*} can be replaced by 
\[
\norme{(t,k) \mapsto \sqrt{r_k^2-1}R(r_ke^{it}, T^*)x^*_k}_{\gamma([0,2\pi)\times\N_N;X^*)} \leq \norme{(x_k^*)_{k \in \N_N}}_{\gamma(\N_N; X^*)}.
\]
Moreover let $(S,\mu)$ be a measure space and let $E(S)$ be a $K$-convex
	Banach function space over $(S,\mu)$
	(see \cite[appendix F]{hnvw} for definition). Then 
	$E(S)$ has finite cotype hence
	according to \cite[Proposition 9.3.8]{hnvw}, there exist $c>0$ and $C>0$ such 
	that for each $f\in \gamma([0,2\pi)\times \N_N; E(S))$,
	\begin{equation}\label{Lattice}
	c\norme{f}_{E(S;L^2([0,2\pi)\times \N_N))} \leq \norme{f}_{\gamma([0,2\pi)\times \N_N;E(S))} \leq C \norme{f}_{E(S;L^2([0,2\pi)\times \N_N))}.
	\end{equation}
	Furthermore, the following equality holds,
	\[
	\norme{f}_{E(S;L^2([0,2\pi)\times \N_N))} = \norme{\Big(\int_{0}^{2\pi} \sum_{k=1}^N|f(\cdot,k)|^2\Big)^{\frac{1}{2}}}_{E(S)}.
	\]
	The space $E(S)^*$ satisfies similar properties.
	Hence using (\ref{Lattice}) and the Khintchine-Maurey inequality 
	\cite[Theorem 7.2.13]{hnvw}, $(ii)$ in Theorem \ref{shifengRdisthm} can be replaced by:
	\begin{itemize}
		\item [$(ii)'$] The spectrum set $\sigma(T)$ is included in $\Dbar$ and  there exists a constant $C>0$ such that for all $N\in \N$, $r_1,\ldots, r_N >1$, $x_1,\ldots,x_N \in X$ and $x^*_1,\ldots,x^*_N \in X^*$ 
		\begin{equation*}
		\norme{\Big(\sum_{k=1}^{N}\int_{0}^{2\pi}(r_k^2-1)|R(r_ke^{it}, T)x_k|^2dt \Big)^{\frac{1}{2}}}_{E(S)} 
		\leq C\norme{\big(\sum_{k=1}^{N}|x_k|^2\big)^{\frac{1}{2}}}_{E(S)}
		\end{equation*}
		and 
		\begin{equation*}
\norme{\Big(\sum_{k=1}^{N}\int_{0}^{2\pi}(r_k^2-1)|R(r_ke^{it}, T^*)x^*_k|^2dt \Big)^{\frac{1}{2}}}_{E(S)^*} 
\leq C\norme{\big(\sum_{k=1}^{N}|x^*_k|^2\big)^{\frac{1}{2}}}_{E(S)^*}.
\end{equation*}
	\end{itemize}
	Thus $T \in B(E(S))$ is power $\gamma$-bounded
	if and only if $(ii)'$ holds true.
	
	In particular the above applies when $E(S)=L^p(S)$ for some $1<p<\infty$.
\end{rq1}

\subsection{Peller calculus}
We denote by $H^p$ the classical Hardy space over the unit disk. Let $\A(\D)$ be the set of all function $F : \overline{\D} \rightarrow \C$ such that there exist two sequences $(f_k)_{k\in \N}$ of $C(\T)$  and $(h_k)_{k\in \N}$  in $H^1$ satisfying
\begin{equation}\label{ineA}
\sum_{k=1}^{\infty} \normeinf{f_k}\norme{h_k}_1 < \infty
\end{equation} 
and
\begin{equation}\label{equA}
\forall z \in \overline{\D }, \quad F(z) = \sum_{k=1}^{\infty} f_k\star h_k(z).
\end{equation} 
We endow $\A(\D)$ with the norm
\[
\norme{F}_{\A} = \inf \big\{\sum_{k=1}^{\infty} \normeinf{f_k}\norme{h_k}_1 \big\}
\]	
where the infimum runs over all sequences  $(f_k)_{k\in \N}$ of $C(\T)$  and $(h_k)_{k\in \N}$ of $H^1$ satifsying \eqref{ineA} and \eqref{equA}. It is known (see \cite[Lemma 3.6.]{pel1}) that with this norm $\A(\D)$ is a Banach algebra for pointwise multiplication.

Let $H$ be a Hilbert space and let $T \in B(H)$ be a power-bounded operator. According to \cite[Theorem 3.5.]{pel1} or \cite[Proposition 4.11]{pis1} there exists $C>0$ such that for each polynomial $P$ one has
\[
\norme{P(T)} \leq C \norme{P}_{\A}.
\]
Now since the set of polynomial is dense in $\A(\D)$ the bounded algebra homomorphism $P \rightarrow P(T)$ extends to a bounded algebra homomorphism from $\A(\D)$ into $B(H)$. Our aim is now to give a $\gamma$-version of this result.

\begin{thm}\label{pispelgamma}
	Let $X$ be a Banach space. Let $T \in B(X)$ be a power $\gamma$-bounded operator. Then the set 
	\[	
	\{P(T): P \text{ is a polynomial with } \norme{P}_{\A}\leq 1 \}
	\]
	is $\gamma$-bounded.
\end{thm}

\begin{proof}
We adapt an argument from \cite[Proposition 4.16.]{pis1}. First we show that the set 
\[
\Gamma = \{ f\star (uv)(T): f\in C(\T), u,v \text { are polynomials}, \normeinf{f}\norme{u}_2\norme{v}_2 \leq 1  \}
\]	
is $\gamma$-bounded. Let $N \in \N$, $f_1, \ldots, f_N \in C(\T)$ and let $u_1,\ldots, u_N$, $v_1,\ldots, v_N$ be polynomials with $\normeinf{f_k}\norme{u_k}_2\norme{v_k}_2 \leq 1 $ for each $k \in \N_N$, and let $x_1,\ldots,x_N \in X$, $x^*_1,\ldots,x^*_N \in X^*$. One has, for each $k\in \N_N$
\[
\langle \big(f_k\star (u_k v_k)\big) (T)x_k, x^*_k \rangle  = \frac{1}{2\pi}\int_{0}^{2\pi} f_k(e^{is})\langle u_k(e^{-is}T)x_k,v_k(e^{-is}T)^*x^*_k \rangle ds .
\]
It follows, using Theorem \ref{inetrace}, that 
\begin{align*}
\sum_{k=1}^{N}|\langle \big(f_k\star &(u_k v_k)\big) (T)x_k, x^*_k \rangle| \leq \frac{1}{2\pi}\sum_{k=1}^{N} \int_{0}^{2\pi} \frac{|f_k(e^{is})|}{\normeinf{f_k}\norme{u_k}_2\norme{v_k}_2}|\langle u_k(e^{-is}T)x_k,v_k(e^{-is}T)^*x^*_k \rangle | ds  \\
& \leq \sum_{k=1}^{N} \int_{0}^{2\pi} \big|\big\langle \frac{u_k(e^{-is}T)x_k}{\norme{u_k}_2},\frac{v_k(e^{-is}T)^*x_k^*}{\norme{v_k}_2} \big\rangle \big| ds \\
& = \norme{(s,k) \mapsto \langle \frac{u_k(e^{-is}T)x_k}{\norme{u_k}_2},\frac{v_k(e^{-is}T)^*x_k^*}{\norme{v_k}_2} \big\rangle }_{L^1((0,2\pi)\times \N_N)} \\
&\leq \frac{1}{2\pi}\norme{ (s,k)\mapsto \frac{u_k(e^{-is}T)x_k}{\norme{u_k}_2}}_{\gamma((0,2\pi)\times \N_N;X)}\norme{ (s,k)\mapsto\frac{ v_k(e^{-is}T)^*x^*_k}{\norme{v_k}_2}}_{\gamma'((0,2\pi)\times \N_N;X^*)}.
\end{align*}
By Lemma \ref{isofourierGamma} and the $\gamma$-boundedness of $\{T^n:n\in \N \}$ one obtains
\begin{align*}
\norme{ (s,k)\mapsto \frac{u_k(e^{-is}T)x_k}{\norme{u_k}_2}}_{\gamma((0,2\pi)\times \N_N;X)} &=  \norme{(n,k)\mapsto c_n(u_k)T^n\Big(\frac{x_k}{\norme{u_k}_2}\Big)}_{\gamma(\N\cup\{0\}\times \N_N;X)} \\
&= \norme{\sum_{n=0}^{\infty}\sum_{k=1}^N \gamma_{n,k}\otimes T^n\Big(\frac{c_n(u_k)x_k}{\norme{u_k}_2}\Big)}_{G(X)} \\
& \leq C_{\gamma}\norme{\sum_{n=0}^{\infty}\sum_{k=1}^N \gamma_{n,k}\otimes \frac{c_n(u_k)x_k}{\norme{u_k}_2}}_{G(X)}.
\end{align*}
For each $k \in \N_N$, $\displaystyle\frac{\norme{(c_n(u_k))_{n\in \N\cup\{0\}}}_{l^2}}{\norme{u_k}_2} = 1$, therefore
\[
\norme{\sum_{n=0}^{\infty}\sum_{k=1}^N \gamma_{n,k}\otimes \frac{c_n(u_k)x_k}{\norme{u_k}_2}}_{G(X)} \leq \norme{(x_k)_{k\in \N_N}}_{\gamma(\N_N;X)}.  
\]	
Then arguing as in the proof of Theorem \ref{shifengRdisthm},
\[
\norme{ (s,k)\mapsto\frac{v_k(e^{-is}T)^*x^*_k}{\norme{v_k}_2}}_{\gamma((0,2\pi)\times \N_N;X^*)} \leq  C_{\gamma}\norme{(x^*_k)_{k\in \N_N}}_{\gamma'(\N_N;X^* )}.
\]
This shows that $\Gamma$ is $\gamma$-bounded.

Now if $P = \sum_{k\in \N }f_k \star h_k$ is a polynomial with $f_k\in C(\T)$, $h_k \in H^1$ and
 $\sum_{k\in \N }\normeinf{f_k}\norme{h_k}_1 \leq 1 $ then using the facts that for each $k\in \N$,
  $h_k$ can be written as a product $h_k = u_kv_k$ with $u_k,v_k \in H^2$ and $\norme{h_k}_1 = \norme{u_k}_2\norme{v_k}_2$ 
  and that the set of polynomials is dense in $H^2$ one obtains that $P \in \overline{absconv (\Gamma)}^{\norme{\cdot}}$.
   Finally one has the following inclusion
	\[	
\{P(T): P \text{ is a polynomial with } \norme{P}_{\A}\leq 1 \} \subset \overline{absconv (\Gamma)}^{\norme{\cdot}},
\]
and the latter set is $\gamma$-bounded thanks to Proposition \ref{propusalgamma}  whence the desired result.
\end{proof}	
\begin{rq1}
Since the set 	$\{P(T): P \text{ is a polynomial with } \norme{P}_{\A}\leq 1 \}$ is $\gamma$-bounded it is uniformly bounded. Therefore, since polynomials are dense in $\A(\D)$, the homomorphism 
\[
u : P \mapsto P(T)
\]
extends to a bounded algebra homomorphism from $\A(\D)$ into $B(X)$. Theorem \ref{pispelgamma} implies  that this homomorphism is $\gamma$-bounded, that is $\{u(f): f \in \A(\D): \norme{f }_{\A} \leq 1 \}$ is $\gamma$-bounded. 
\end{rq1}

\bibliographystyle{plain}
\bibliography{article}

\end{document}